\documentclass[12pt]{elsarticle}

\usepackage{amsfonts}
\usepackage{amsmath}
\usepackage{graphicx}
\usepackage{epsfig}
\usepackage{amssymb}
\usepackage{amstext}
\usepackage{amsfonts}
\usepackage{amsmath}
\usepackage{graphicx}
\usepackage{epsfig}
\usepackage{amssymb}

\newtheorem{theorem}{Theorem}[section]
\newtheorem{lemma}{Lemma}[section]
\newtheorem{remark}{Remark}[section]
\newtheorem{corollary}{Corollary}[section]
\newtheorem{definition}{Definition}[section]
\newtheorem{example}{Example}[section]

\newproof{proof}{Proof}

\usepackage{tikz}

\usetikzlibrary{calc}

\def\afrac#1#2{\ifinner {#1}/{#2} \else \frac{#1}{#2} \fi}
\def\1#1{\mbox{\rm{#1}}}

\begin{document}
\begin{frontmatter}

\title
{Orbits and Hamilton bonds in a family of plane triangulations with vertices of degree three or six}
\author{Jan~Florek}
\ead{jan.florek@ue.wroc.pl}

\address{Institute of Mathematics,
 University of Economics,\\  53--345 Wroc{\l}aw, ul. Komandorska 118/120,  Poland}

\begin{abstract}
Let $\cal{P}$ be the family of all $2$-connected plane triangulations with vertices of degree three or six. Gr\"{u}nbaum and Motzkin proved (in the dual terms) that every graph $P \in \cal{P}$ is factorable into factors $P_0$, $P_1$, $P_2$ (indexed by elements of the cyclic group $Q = \{0,1,2\}$) such that every factor $P_q$ consists of two induced paths with the same length $M(q)$, and $K(q)-1$ induced cycles with the same length $2M(q)$. For $q \in Q$, we define an integer $S^+(q)$  such that the vector $(K(q), M(q), S^+(q))$  determines the graph $P$ (if $P$ is simple) uniquely up to orientation-preserving isomorphism. We establish arithmetic equations that will allow calculate the vector $(K(q+1), M(q+1), S^+(q+1))$ by the vector $(K(q), M(q), S^+(q))$, $q \in Q$. We present some applications of the equations. The set $\{(K(q), M(q), S^+(q)): q \in Q\}$ is called the orbit of $P$. We characterize one point orbits of graphs in $\cal{P}$. We prove that if $P$ is of order $4n +2$, $n \in\mathbb{N}$, than it has a Hamilton bond such that the end-trees of the bond are equitable 2-colorable and have the same order. We prove that if $M(q)$ is odd and $K(q) \geqslant \frac{M(q)}{3}$, then there are two disjoint induced paths of the same order, which vertices together span all of $P$.

\end{abstract}
\end{frontmatter}

\section{Introduction}

Let $G_i$, i = 1, 2, be a plane graph with the vertex set $V(G_i)$, the edge set $E(G_i)$, and the face set $F(G_i)$. An isomorphism $\sigma$ between $G_1$ and $G_2$ is called \textsl{combinatorial} if it can be extended to a bijection
$$\sigma: {V(G_1)}\cup{E(G_1)}\cup{F(G_1)} \rightarrow {V(G_2)}\cup{E(G_2)}\cup{F(G_2)}$$
that preserves incidence not only of vertices with edges but also of vertices and edges with faces (Diestel \cite[p.~93]{flo2}). Furthermore, we say that $G_1$~and~$G_2$ are \textsl{op-equivalent} (equivalent up to orientation-preserving isomorphism) if $\sigma$ is a combinatorial isomorphism which preserves the counter clockwise orientation. (Formally: we require that $g_1,g_2,g_3$ are counter clockwise successive edges incident with a vertex $v$ if and only if $\sigma(g_1), \sigma(g_2), \sigma(g_3)$ are counter clockwise successive edges incident with $\sigma(v)$).

A \textsl{factor} of a graph is a subgraph whose vertex set is that of the whole graph. A graph $H$ is said to be \textsl{factorable} into factors $H_1$, $H_2$, \ldots, $H_t$ if these factors are pairwise edge-disjoint and
$E(H) = E(H_1) \cup {}\ldots {}\cup E(H_t)$
(see Chartrand-Lesniak \cite[p.~246]{flo1}). An edge (respectively a subgraph) of $H$ is said to be of \textsl{class} $q$ if this edge (respectively any edge of this subgraph) belongs to the factor $H_q$.

Let $\cal{P}$ be the family of all $2$-connected plane triangulations all of whose vertices are of degree $3$ or $6$, and suppose that $P\in\cal{P}$. Gr\"{u}nbaum and Motzkin  \cite[Lemma 2]{flo6} proved (in the dual terms) that the graph $P$ is factorable into factors $P_0$, $P_1$, $P_2$ (indexed by elements of the cyclic group $Q = \{0,1,2\}$) satisfying the following condition:
\begin{description}
\item{(*)} for every vertex $v$ in $P$, if two edges are counter clockwise successive edges incident with $v$, then  the first edge is of  class $q$ and  the second one is of class $q+1$, for some $q \in Q$.
\end{description}
Hence, it follows that $P_q$, $q \in Q$, consists of two induced paths with the same length $M(q)$, and $K(q)-1$ induced cycles with the same length $2M(q)$. More precisely,
\begin{description}
\item{(**)} for $q \in Q$, there is a drawing of $P$ (called \textsl{$q$-drawing}) which is op-equivalent to $P$. The $q$-drawing of $P$ consists of a maximal path of class $q$ with the length $M(q)$, and this path is surrounded by $K(q) - 1$ disjoint cycles of class $q$ with the same length $2M(q)$. Finally, there is another maximal path of class $q$ with the length $M(q)$ (called \textsl{outer path}) around the outside of the last cycle (see Example 1.1).
\end{description}
By (**) we have the following Gr\"{u}nbaum and Motzkin result \cite[Theorem 2]{flo6}:
\\[6pt]
$
\begin{array} {ll}
{\rm (1)} & 2K(q)M(q)+2 \quad \mbox{is the order of $P$}.
\end{array}
$
\\[6pt]
Notice that the outer path may be added at different positions. We define (Definition~2.2) the integer $0\leqslant S^+ (q) < M(q)$ (and also $0 < S^- (q) \leqslant M(q)$) that determines the  position.
In Theorem~\ref{theo2.1} we show the following relation between $S^+ (q)$ and $S^- (q)$:
\\[6pt]
$
\begin{array} {ll}
{\rm (2)} & S^- (q) - S^+(q) \equiv K(q) \pmod {M(q)} .
\end{array}
$
\\[6pt]
The vector $(K(q), M(q), S^+(q))$, for $q\in Q$, is called a $q$-\textsl{index-vector} of $P$, and the set $\{(K(q), M(q), S^+(q)) :  q\in Q \}$ is called the \textsl{orbit} of $P$.
The purpose of this article is to establish arithmetic equations that will allow to calculate the $(q+1)$-index-vector by the $q$-index-vector of $P$. In Theorem~\ref{theo3.1} we prove the following equality:
\\[6pt]
$
\begin{array} {ll}
{\rm (3)} &  K(q+1) = |S^+ (q), M(q)|,
\end{array}
$
\\[6pt]
where $|s,m|$ is the greatest common divisor of integers $s \geqslant 0$ and $m \geqslant 1$ ($|0, m| = m$). Let $0 < b \leqslant \afrac{M(q)}{|S^+ (q), M(q)|}$ be an integer such that $bS^+ (q) \equiv -|S^+ (q), M(q)| \pmod{M(q)}$. In Theorem~\ref{theo3.2} we prove that
\\[6pt]
$
\begin{array}{ll}
{\rm (4)} & S^- (q+1) = b\, K(q) \, .
\end{array}
$
\\[6pt]
Notice that every isomorphism between two $3$-connected plane graphs is combinatorial (Diestel \cite[p.~94] {flo2}). Hence, if $P$ is simple, then it is determined by any of its index-vectors uniquely up to op-equivalence.
Therefore, using equations \1{(1)--(4)} we can verify whether considered simple graphs in $\cal{P}$ are op-equivalent.
\begin{example}\label{exam1.1}
Let us consider the simple graph $S_0$ of\/ \1{Fig.~1}, and simple graphs  $S_1$, $S_2$ of\/ \1{Fig.~2}. We assume that edges $g_0$, $g_1$, $g_2$ are of class $0$, $1$,~$2$, respectively, and they are incident with a common vertex of degree $3$. The graph $S_0$ has the $0$-index-vector $(1,6,3)$, $S_1$ has the $1$-index-vector $(3,2,0)$ and $S_2$ has the $2$-index-vector $(2,3,1)$. Using equations \1{(1)--(4)} we check that $\{(1,6,3),(3,2,0),(2,3,1)\}$ is their common orbit \1{(}see \1{Example~3.1)}.
Hence, the graph $S_0$ is the $0$-drawing of the graph  $S_1$ and $S_2$, $S_1$ is the $1$-drawing of $S_0$ and $S_2$, and $S_2$ is the $2$-drawing of $S_0$ and $S_1$. Therefore these graphs are op-equivalent.
\end{example}

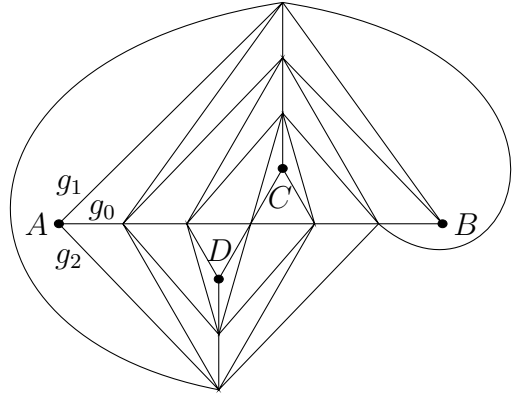
\begin{figure}
\centering

\begin{tikzpicture}[xscale=0.85, yscale=0.7361]

\coordinate  (A) at (0,0);
\coordinate  (B) at (6,0);

\coordinate  (E) at (1,0);
\coordinate  (F) at (2,0);
\coordinate  (G) at (3,0);
\coordinate  (H) at (4,0);
\coordinate  (J) at (5,0);

\coordinate (C) at (3.5,1);
\coordinate (K) at (3.5,2);
\coordinate (L) at (3.5,3);
\coordinate (M) at (3.5,4);

\coordinate (D) at (2.5,-1);
\coordinate (N) at (2.5,-2);
\coordinate (P) at (2.5,-3);

\draw (A)--node[near start,above,xshift=-7mm,yshift=-1mm]{$g_0$}(B);
\draw (C)--(M);
\draw (D)--(P);

\draw (A)--node[near start,above,xshift=-6mm,yshift=-5mm]{$g_1$}(M)--(E)--(L)--(F)--(K)--(G)--(C)--(H)--(K)--(J)--(L)--(B)--(M);
\draw (A)--node[near start,above,xshift=-4mm,yshift=-2mm]{$g_2$}(P)--(E)--(N)--(F)--(D)--(G)--(N)--(H)--(P)--(J);

\draw (J) .. controls (7,-2) and (9,3) .. (M);
\draw (M) .. controls (-2,3) and (-2,-2) .. (P);

\filldraw[black]
(A) circle (2pt)
(B) circle (2pt)
(C) circle (2pt)
(D) circle (2pt);

\draw (A) node[anchor=east]{$A$};
\draw (B) node[anchor=west]{$B$};
\draw (C) node[anchor=north, yshift=-1mm, xshift=-1pt]{$C$};
\draw (D) node[anchor=south, yshift=1mm, xshift=0.1mm]{$D$};

\end{tikzpicture}

\caption{The graph $S_0$ with the $0$-index-vector $(1,6,3)$.}
\end{figure}

\begin{figure}
\centering

\begin{tabular}[b]{c@{\hspace{-8mm}}c}

\begin{tikzpicture}[xscale=0.85, yscale=0.7361]

\coordinate (O) at (0,0);
\coordinate  (A) at (-1,0);
\coordinate  (B) at (1,0);

\coordinate (C) at (2.5,2.5);
\coordinate (D) at (-2.5,-2.5);
\coordinate (CC) at (3.5,3.5);

\coordinate  (E) at (-2,0);
\coordinate  (F) at (-2.5,0);
\coordinate  (G) at (2,0);
\coordinate  (H) at (2.5,0);

\coordinate (I) at (-0.5,1);
\coordinate (J) at (0.5,1);
\coordinate (K) at (-0.5,-1);
\coordinate (L) at (0.5,-1);
\coordinate (M) at (0,2);
\coordinate (N) at (0,-2);

\draw (A)--node[near start,below,xshift=-2mm,yshift=2.1mm]{$g_0$}(N)--(L)--(K)--(O)--node[above,xshift=-0.1mm,yshift=-1mm]{$g_1$}(A)--node[above,xshift=-2mm,yshift=-1.9mm]{$g_2$}(I)--(L)--(B)--(O)--(J)--(I)--(M)--(B);
\draw (I) .. controls (-2,1) and (-2,-1) ..(K);
\draw (J) .. controls (2,1) and (2,-1) ..(L);

\draw (I) .. controls (-1,1) and (-2,0.5) ..(F);
\draw (K) .. controls (-1,-1) and (-2,-0.5) ..(F);
\draw (J) .. controls (1,1) and (2,0.5) ..(H);
\draw (L) .. controls (1,-1) and (2,-0.5) ..(H);

\draw (H) .. controls (2.5,1) and (1,2) .. (M);
\draw (H) .. controls (2.5,-1) and (1,-2) .. (N);
\draw (F) .. controls (-2.5,1) and (-1,2) .. (M);
\draw (F) .. controls (-2.5,-1) and (-1,-2) .. (N);

\draw (CC) -- (H) -- (C) -- (CC) -- (M) -- (C);
\draw (N) -- (D) -- (F);

\draw (F) .. controls (-2.5,3) and (3,4) .. (CC);
\draw (N) .. controls (4,-2.5) and (4,3.5) .. (CC);
\draw (D) .. controls (4,-4.5) and (5,3.5) .. (CC);

\filldraw[black]
(A) circle (2pt)
(B) circle (2pt)
(C) circle (2pt)
(D) circle (2pt);

\draw (A) node[anchor=east]{$A$};
\draw (B) node[anchor=west]{$B$};
\draw (C) node[anchor=west]{$C$};
\draw (D) node[anchor=east]{$D$};

\end{tikzpicture}
 &

\begin{tikzpicture}[xscale=0.85,yscale=0.7361]

\coordinate (O) at (0,0);
\coordinate  (A) at (-1,0);
\coordinate  (B) at (0,-2);
\coordinate (C) at (2,0);
\coordinate (D) at (1,2);

\coordinate  (BB) at (0,-3);
\coordinate (DD) at (1,3);

\coordinate  (E) at (1,0);

\coordinate  (F) at (-0.5,1);
\coordinate  (G) at (0.5,1);
\coordinate  (H) at (1.5,1);

\coordinate  (I) at (-0.5,-1);
\coordinate  (J) at (0.5,-1);
\coordinate  (K) at (1.5,-1);

\coordinate (fake1) at (0,3.99);
\coordinate (fake2) at (0,-4.5);

\draw (A)--node[near start,above,xshift=-1mm,yshift=-0.6mm]{$g_0$}(F)--(O)--node[above,xshift=-0.1mm,yshift=-1mm]{$g_2$}(A)--node[below,xshift=-1.5mm,yshift=1.9mm]{$g_1$}(I)--(O)--(G)--(E)--(O)--(J)--(E)--(H)--(C)--(E)--(K)--(C);
\draw (DD)--(F)--(H)--(DD)--(G)--(D)--(H)--(DD)--(D);
\draw (I)--(BB)--(K)--(I)--(B)--(J)--(BB)--(B);

\draw (F) .. controls (-2,1) and (-2,-1) ..(I);
\draw (H) .. controls (3,1) and (3,-1) ..(K);

\draw (F) .. controls (-3,1) and (-2.5,-3) ..(BB);
\draw (K) .. controls (4,-1) and (3,3) ..(DD);
\draw (BB) .. controls (4,-2) and (4,3) ..(DD);

\filldraw[black]
(A) circle (2pt)
(B) circle (2pt)
(C) circle (2pt)
(D) circle (2pt);

\draw (A) node[anchor=east]{$A$};
\draw (B) node[anchor=south,yshift=1mm]{$B$};
\draw (C) node[anchor=west]{$C$};
\draw (D) node[anchor=north,yshift=-1mm]{$D$};

\filldraw[white] (fake1) circle (0.1pt);
\filldraw[white] (fake2) circle (0.1pt);

\end{tikzpicture}
\end{tabular}

\caption{The graph $S_1$ with the $1$-index-vector $(3,2,0)$, and $S_2$ with the $2$-index-vector $(2,3,1)$.}
\end{figure}
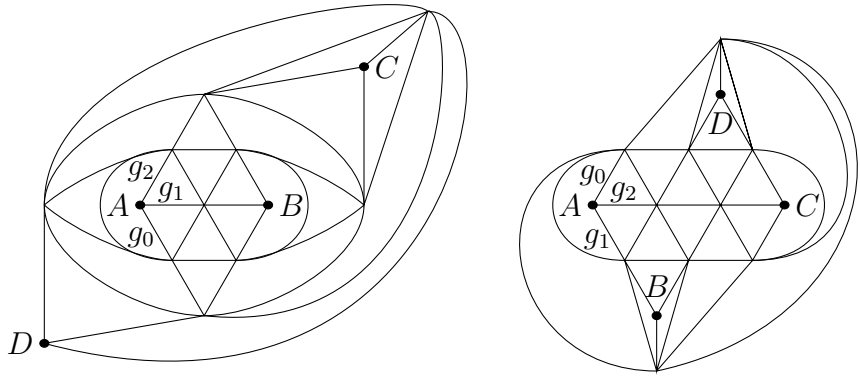

We are going to present some applications of equations \1{(1)--(4)}. From (**) follows that $X = \{ (k,m,s) \in \mathbb{\mathbb{Z}}^3 :  1\leqslant k\, , 1 \leqslant m\, , 0 \leqslant s < m \}$ is the set of all index-vectors of graphs in $\cal{P}$. In Theorem~\ref{theo4.1} we  characterize one point orbits: $(k, m, s) \in X$  is a one point orbit of a graph in $\cal P$ if and only if $ m = kn$, $s = kx$, where $k$, $x$, $n$ are integers such that $k \geqslant 1$, $0 \leqslant x < n $ and $n$ is a divisor of $x^2 + x + 1$.
By Schinzel and Sierpi\'{n}ski \cite{flo8} the set of all integral solutions of
the equation $x^2 +  x+ 1 = 3 y^2$ is infinite. It follows that there is an infinite family of graphs in ${\cal P}$ with one point orbit.

If $P$ has an index-vector $(K(q), M(q), S^+(q))$, then its mirror reflection has the index-vector $(K(q), M(q), M(q)-S^-(q))$. We say that $P$ is \textsl{double mirror symmetric} if there exists $q_1, q_2 \in Q$ such that  $S^+ (q_i) = M(q_i)-S^-(q_i)$, for $i = 1, 2$. In Theorem~\ref{theo4.2} we show that $P$ is double mirror symmetric if and only if $P$ has a one point orbit of the form $\{ (k,k,0) \}$ or $ \{ (k,3k,k) \}$ for some $k\in \mathbb{N}$.

A \textsl{bond} of a connected plane graph $G$ is a minimal non-empty edge cut (Diestel \cite[p.~25]{flo2}). A \textsl{Hamilton bond} of $G$ is a bond $B$ such that both components of $G \backslash B$ are trees, and the trees are called \textsl{end-trees} of the bond. It is known that a Hamilton bond in $G$ is the algebraic dual of a~Hamilton cycle in the dual graph (see Stein \cite{flo10}). Goodey  \cite{flo4} showed that every  $2$-con\-nect\-ed cubic plane graph whose faces are only triangles or hexagons has a Hamilton cycle. Hence, every graph in ${\cal P}$ has a Hamilton bond. In Theorem~\ref{theo5.1} we prove that for every Hamilton bond of a graph in $ \cal{P}$ the end-trees of the bond have the same order. In Theorem~\ref{theo5.2} we prove a similar result for the family $\cal{H}$ of all $2$-con\-nect\-ed plane triangulations all of whose vertices are of degree at most~$6$. Namely, for every Hamilton bond of a graph in $ \cal{H}$ the orders of end-trees of the bond differ by at most~$3$.

In \cite{flo7} Meyer introduced the following notation of equitable colorability. A graph $G$ is  \textsl{equitable $k$-colorable} if there exists a proper $k$-coloring of $G$ such that the size of any two color classes differ by at most one. It is easy to see, by condition (**), that every graph in $\cal{P}$ is equitable $4$-colorable. We know that every graph in ${\cal P}$ has a Hamilton bond and the end-trees of the bond have the same order. One may guess that every graph in $\cal{P}$ has a Hamilton bond such that end-trees are equitable $2$-colorable. In fact, in Theorem 6.1 we prove (using the equations {(1)--(3)}) that this is the case if  $P$ is of order $4n+2$, $n \in \mathbb{N}$. We also prove that if $P$ has an index-vector $(K(q), M(q), S^+(q))$ such that $M(q)$ is odd and $K(q) \geqslant \frac{M(q)}{3}$, then there are two disjoint induced paths which vertices together span all of $P$ (see Theorem~\ref{theo6.2}).

\section{Index-vector}
Let $\cal{P}$ be the family of all $2$-connected plane triangulations all of whose vertices are of degree~$3$ or~$6$. Fix $P \in \cal{P}$. Let $P$ be factorable into factors $P_0$, $P_1$, $P_2$ (indexed by elements of the cyclic group $Q = \{0,1,2\}$) satisfying the condition (*). We recall that a subgraph of $P$ is said to be of \textsl{class} $q\in Q$ if any edge of the subgraph belongs to the factor $P_q$. Let $M(q)$ be the length of a maximal path of class $q$, and $K(q)$ the distance between the two maximal paths of this class in $P$.
\begin{definition}\label{def2.1}
Let $A$ be a vertex of degree $3$ in the graph $P$, and suppose that $[A, q]$ is a maximal path of class $q$ with a fixed orientation $v_0 v_1 \ldots \ v_{M(q)}$ such that $A = v_0$ is its {\em initial} and $A_q = v_{M(q)}$ is its {\em terminal} vertex. An edge $e$ adjacent to the path $[A, q]$ is called a \textsl{left branch} of the path if it is branching off from $[A, q]$ to the left \1{(}more precisely, if  $v_j v_{j+1}, e$, $0\leqslant j < M(q)$, or $e, v_{j-1}v_j$, $0 < j \leqslant M(q)$,
are counter clockwise successive edges incident with the vertex $v_j$\1{)}. Otherwise, it is called  a \textsl{right branch} of the path. We put
$$[ A, q] (e) = \left\{\begin{array}{ll} j & \quad \mbox{if $e$ is a left branch of\/ $[A, q ]$},
\\[4pt]
 2M(q) - j & \quad \mbox{if $e$ is a right branch of\/ $[A, q ]$}.
\end{array}\right.
$$
\end{definition}
\begin{remark}\label{rem2.1}
Notice that $[A, q] = v_0 v_1 \ldots \ v_{M(q)}$ if and only if $[A_q,q]=  v_{M(q)} v_{M(q)-1} \ldots \ v_0$. An edge $e$ is a left branch of the path $[A,q]$ if and only if it is a right branch of the path $[A_q,q]$. Moreover, we have
$$
\left| [A_q, q ] (e) - [ A, q ] (e) \right| = M(q) .
$$
\end{remark}
\begin{lemma}\label{lem2.1}
Let $A, C$ be ends of two different maximal paths of class $q$.
\begin{description}
\item \1{(1)} If $e,\hat{e}$ and $f,\hat{f}$ are pairs of end-edges  of  two minimal paths of class $q+1$ so that $e, f$ are adjacent to the path $[A, q]$ and  $\hat{e}, \hat{f}$ are adjacent to the path $[C, q]$, then
$$
[A, q ] (f) - [A, q ] (e) \equiv [C,q] (\hat{e})-[C,q] (\hat{f})\  \pmod {2M(q)} .
$$
\item \1{(2)} Moreover, if the edge $e$ is incident with $A$, and the edge $\hat{f}$ is incident with $C$, then
$$
[A, q ] (f) = [C, q](\hat{e}).
$$
\end{description}
\end{lemma}
\begin{proof}
It is more clear when we consider the $q$-drawing of $P$. Notice that
$$
[A', q ] (f) - [A', q ] (e) = [C',q] (\hat{e})-[C',q] (\hat{f})
$$
for some $A' \in \{A, A_q\}$ and for some  $C' \in \{C,  C_q\}$. Hence, by Remark 2.1 we obtain (1). By (*), $e$ is a left branch of the path $[A,q]$ and $\hat{f}$ is a left branch of the path $[C, q]$. Hence, $[A, q ](e) = 0$ and $[C, q](\hat{f}) = 0$, which yields (2).
\qed
\end{proof}
\begin{definition}\label{def2.2}
Let $A$, $C$ be ends of two different maximal paths of class $q$ in the graph $P$, and suppose $f$
\1{(}or $g$\1{)} is the first edge of the path $[C, q+1 ]$ \1{(}or $[C, q-1]$, respectively\1{)}
which is adjacent to the path $[A, q]$. Let
$$S^+ (q) = \left\{\begin{array}{ll}
[A, q ] (f),& \quad \mbox{if $f$  is a left branch of $[A, q ]$},
\\[4pt]
[A, q ] (f)-M(q),& \quad \mbox{if $f$ is a right branch of $[A, q ]$},
\end{array}\right.
$$
$$S^- (q) = \left\{\begin{array}{ll}
[A, q ] (g),& \quad \mbox{if $g$  is a left branch of $[A, q ]$},
\\[4pt]
[A, q ] (g)-M(q),& \quad \mbox{if $g$ is a right branch of $[A, q ]$},
\end{array}\right.
$$
\end{definition}

Notice that by Remark~2.1 and Lemma~2.1(2) the definition of $S^+ (q)$ and $S^- (q)$ do not depend on the choice of ends of two different maximal paths of class $q$. The following theorem shows that $S^+ (q)$ is determined by $S^- (q)$ and vice versa.

\begin{theorem}\label{theo2.1}
$$
S^- (q) - S^+(q) \equiv K(q) \pmod {M(q)} .
$$
\end{theorem}

\begin{proof}
It is more clear if one considers the $q$-drawing of $P$. Let $A$, $C$ be ends of two different maximal paths of class $q$ in $P$, and suppose $f$(or $g$) is the first edge of the path $[C, q+1 ]$ (or $[C, q-1]$) which is adjacent to the path $[A,q]$, say in a vertex $E$ (or $F$, respectively).  If $V$ is the last common vertex of the path $[C, q+1]$ with a segment $CF$ of $[C, q-1]$, then we have
$$
 [A, q] (g) - [ A, q] (f) \equiv |VE| \equiv K(q) \pmod{ 2M(q)}.
$$
Hence,
$$
S^- (q) - S^+(q) \equiv [A, q] (g) - [A, q] (f) \equiv \ K(q)  \pmod{ M(q)},
$$
which completes the proof.
\qed
\end{proof}
\section{Billiards and structure of graphs in  $\cal{P}$}

Let $\cal{P}$ be the family of all $2$-connected plane triangulations all whose vertices are of degree~$3$ or~$6$. Fix $P \in \cal{P}$ and $q \in Q$ (where $Q = \{0,1,2\}$ is the cyclic group). Let $(K(q), M(q), S^+(q))$ be the $q$-index-vector of $P$.

If $0 < \theta < 1$, then a \textsl{$\theta$-billiard sequence} is a sequence $F(j) \in [0,1)$, $j \in N$, which satisfies the following conditions (see \cite{flo3}): $F(1) = 0$  and
$$
F(j) + F(j+1) = \left\{\begin{array}{ll}
\theta \quad \hbox{or} \quad 1+\theta, & \hbox{ for an odd }  j  ,
\\[4pt]
0 \quad \hbox{or} \quad 1,  &  \hbox{ for an even }  j  .
\end{array}\right.
$$

We consider a billiard table rectangle with perimeter of length $1$ with the bottom left vertex
labeled $v_0$, and the others, in a clockwise direction $v_1,v_2$ and $v_3$.
The distance from $v_0$ to $v_1$ is $\theta/2$. We describe the position of
points on the perimeter by their distance along the perimeter measured
in the clockwise direction from $v_0$, so that $v_1$ is at position $\theta/2$,
$v_2$ at $1/2$ and $v_3$ at $(\theta+1)/2$. If a billiard ball is
pushed from position $F(1)= 0$ at the angle of $\pi/4$, then it will rebound against
the sides of the rectangle consecutively at points $F(2)$, $F(3)$, \dots\ .

The following Lemma~3.1 comes from  \cite[Theorem~3.3(2) and Example~3.1]{flo3}.
\begin{lemma}\label{lem3.1}
If $0 < \afrac{s}{m} < 1$ is a fraction, $d = |s, m|$ and \ $F(j), j \in N,$ is the $\afrac {s}{m}$-billiard sequence, then,
 \\[6pt]
 \1{(1)} $\{2mF(1), 2mF(2),\ldots, 2mF(\afrac{m}{d})\} = \{0, 2d, 4d, \ldots, 2m-2d\}$.
   \\[6pt]
 \1{(2)}
 $2m F(\afrac{m}{d}) = \left\{\begin{array}{ll}
 s,  & \mbox{for  $\afrac{s}{d}$ even},
 \\[4pt]
 m,  & \mbox{for  $\afrac{m}{d}$ even},
 \\[4pt]
 s+m, & \mbox{for $\afrac{s}{d}$ and $\afrac{m}{d}$ both odd},
 \end{array}\right.
 $
 \\[4pt]
\indent
and \ $2m F(j) \notin \{s, m, s+m\}$, \  for $1 \leqslant j < \afrac{m}{d}$.
 \\ [6pt]
 \1{(3)} If $a, b$  are natural numbers, $am - bs = d$ and \ $b \leqslant \afrac{m}{d}$, then,
 \\ [6pt]
 \indent
$ 2m F(b) = \left\{ \begin{array}{ll}
 s + d,   & \mbox{for  $a$ even},
 \\[4pt]
 m  - d,  & \mbox{for  $b$  even},
 \\[4pt]
 s+m + d, & \mbox{for $a$  and  $b \neq 1$ both odd},
 \\[4pt]
 0, & \mbox{for   $a=b=1$}.
 \end{array}\right.
 $
\end{lemma}

\begin{remark}\label{rem3.1}
The sequence of all reduced fractions of the interval $[0,1]$ with denominators not exceeding $n$, listed in order of their size, is called the Farey sequence of order $n$
\1{(}$\afrac{0}{1}$ is the smallest and $\afrac{1}{1}$ the biggest fraction of any Farey sequence\1{)}.
Let $0 \leqslant \afrac{s}{m} < 1$ be a fraction, and suppose that $\afrac{s'}{m'} = \afrac{s}{m} $ is a fraction in lowest terms. Then $\afrac{s'}{m'} < \afrac{a}{b}$ are consecutive fractions in the Farey sequence of order $m'$ \/ if and only if \/ $a m - bs = |s, m|$ and $b \leqslant m'$  \1{(}see Schmidt \cite{flo9}\/\1{)}.
\end{remark}
The following theorem shows that the structure of the graph $P$ is closely related to the $\afrac{S^+(q)}{M(q)}$-billiard sequences.

\begin{theorem}\label{theo3.1}
Let $A$ be a vertex of degree $3$ in $P$, and suppose that $e_1$, $e_2$, \ldots, $e_n$ is a sequence of all consecutive edges of the path $[A, q+1] $ which are adjacent to the path $[A, q]$.
\\[6pt]
$
\begin{array}{ll}
{\rm (1)} \ \hbox{If} \ n > 1 , \ \hbox{then}
\end{array}
$
$$
[A, q](e_j) = 2M(q)F(j), \hbox{ for } 1 \leqslant j \leqslant n,
 $$
\indent
where $F(j)$, $j \in \mathbb{N}$, is  the $\afrac{S^+ (q)}{M(q)}$-billiard sequence,
\\ [6pt]
$
\begin{array}{lll}
{\rm (2)}& n = \afrac{M(q)}{|S^+ (q), M(q)|},
\\ [6pt]
{\rm (3)}& K(q+1) = |S^+ (q), M(q)|,
\\ [6pt]
{\rm (4)}& K(q+ 1) M(q+ 1) = K(q) M(q).
\end{array}
$
\end{theorem}
\begin{proof}
Since $2K(q) M(q) +2$ is the order of $P$, condition (4) holds.

If $n =1$, then $S^+ (q) = 0$, $M(q+1) = K(q)$ and, by (4), $K(q+1) = M(q)$.
Hence, conditions (2) and (3) are satisfied.

Let $n > 1$. Let $C$ be a vertex of degree $3$, $C \neq A$, $C \neq A_q$, and suppose that $f$ is the first edge of the path $[C,q+1]$ which is adjacent to the path $[A,q]$. Without loss of generality we can assume, by Remark 2.1, that $f$ is a left branch of $[A,q]$. Hence, $[A, q](f) =  S^+ (q)$. Suppose that $\hat{e}_1$, $\hat{e}_2$, \ldots , $\hat{e}_n$ is a sequence of all consecutive edges of the path $[A, q+1]$ which are adjacent to the path $[C, q]$. Note that the edges $\hat{e}_{2j-1}$, $\hat{e}_{2j}$ are
incident with the same vertex of the path $[C, q]$ and that they are on the opposite sides of this path. Hence, we get
$$
[C, q] (\hat{e}_{2j-1} ) + [C, q] (\hat{e}_{2j} ) = 2M(q).
$$
By Lemma~2.1(1), we have
$$[A, q](e_j) + [C, q ](\hat{e}_j) \equiv [A, q](f) \equiv S^+ (q) \pmod{2M(q)}, \quad \hbox{for} \ 1 \leqslant  j \leqslant n.$$
Hence, we obtain
\begin{eqnarray*}
[A, q](e_{2j-1}) + [A, q](e_{2j} ) \equiv 2 S^+ (q) \pmod{2M(q)}, \ 2\leqslant 2j \leqslant n.
\end{eqnarray*}
Since $0 \leqslant [A, q] (e_{2j-1}) + [A, q] (e_{2j}) < 4M(q)$ and $0 < S^+ (q) < M(q)$ we get
\\[6pt]
$
\begin{array}{ll}
{\rm(i)} \quad [A, q](e_{2j-1}) + [A, q](e_{2j} ) =
2S^+ (q)\ \ \hbox{or} \ \ 2M(q) + 2S^+ (q), \ 2\leqslant 2j \leqslant n.
\end{array}
$
\\[6pt]
By analogy, the edges $e_{2j}$, $e_{2j+1}$ are incident with the same vertex of the path $[A, q]$ and, therefore, they are on the opposite sides of this path. Hence, we have
\\[6pt]
$
\begin{array}{ll}
{\rm(ii)} \quad [A, q](e_{2j}) + [A, q](e_{2j+1} ) = 2M(q), \  \ 2\leqslant 2j \leqslant n-1 .
\end{array}
$
\\[6pt]
By (i) and (ii) we obtain (1).

By definition of $A_q$, $C$ and $C_q$, we have
\begin{eqnarray*}
A_{q+1} = A_q  \hbox{ and } j = n &\Leftrightarrow& [A,q](e_j)= M(q),
\\[4pt]
A_{q+1} = C \hbox{ and } j = n &\Leftrightarrow & [C,q](\hat{e}_j) = 0  \Leftrightarrow [A,q](e_j) = S^+(q),
\\[4pt]
A_{q+1} = C_q  \hbox{ and } j = n &\Leftrightarrow & [C,q](\hat{e}_j) = M(q) \Leftrightarrow [A,q](e_j) = M(q) + S^+(q).
\end{eqnarray*}
Accordingly,
\\[6pt]
$
\begin{array}{ll}
{\rm(iii)} & \quad [A, q](e_n) \in \{ M(q), S^+ (q), M(q) + S^+ (q)\}, \quad \hbox{and}
\\[4pt]
 & \quad [A, q](e_j) \notin \{ M(q), S^+ (q), M(q) + S^+ (q)\}, \quad \hbox{for} \ j < n.
\end{array}
$
\\[6pt]
By (1) and Lemma~3.1(2), condition (iii) leads to $n = \afrac{M(q)}{|S^+ (q), M(q)|}$.

Since $n = \afrac{M(q)}{|S^+ (q), M(q)|}$ condition (4) shows that
$$
M(q+1) = n K(q) = \afrac{M(q) K(q)}{|S^+ (q), M(q)|} = \afrac{M(q+1) K(q+1)}{|S^+ (q), M(q)|} .
$$
Thus $K(q+1) = |S^+ (q), M(q)|$ and condition (3) holds.
\qed
\end{proof}
By analogy, we obtain the following corollary:
\begin{corollary}\label{corollary3.1}
Let $A$ be a vertex of degree $3$ in $P$, and suppose that $e_1$, $e_2$, \ldots, $e_n$ is a sequence of all consecutive edges of the path $[A, q-1] $ which are adjacent to the path $[A, q]$.
\\[6pt]
$
\begin{array}{ll}
{\rm (1)} \ \hbox{If} \ n > 1 , \ \hbox{then}
\end{array}
$
$$
[A, q](e_j) = 2M(q)F(j), \hbox{ for } 1 \leqslant j \leqslant n,
 $$
\indent
where $F(j)$, $j \in \mathbb{N}$, is  the $\afrac{S^- (q)}{M(q)}$-billiard sequence,
\\ [6pt]
$
\begin{array}{lll}
{\rm (2)}& n = \afrac{M(q)}{|S^- (q), M(q)|},
\\ [6pt]
{\rm (3)}& K(q-1) = |S^- (q), M(q)|.
\end{array}
$
\end{corollary}
\begin{theorem}\label{theo3.2}
Let $A$ be a vertex of degree $3$ in $P$, and suppose that $a$, $b$ are natural numbers such that $a M(q) - bS^+ (q) = d$ and  $b \leqslant \afrac{M(q)}{d}$, where $d = |S^+ (q), M(q)|$. Then we have:
\\[6pt]
$
\begin{array}{ll}
{\rm (1)} & S^- (q+1) = b  K(q),
\\ [6pt]
{\rm (2)} & S^+ (q+1) \equiv b K(q) - K (q+1) \pmod{ M(q+1)}.
\\[6pt]
\end{array}
$
\end{theorem}
\begin{proof}
 Suppose that $ e_1$, $e_2$, \ldots, $e_n $ is a sequence of all consecutive edges of the path $[A, q+1]$ which are adjacent to the path $[A, q]$ in vertices $A = E_1$, $E_2$, \ldots, $E_n$, respectively.

 If $n = 1$, then  $A$  is the only common vertex of paths $[A,q+1]$ and $[A, q]$. Hence,
$S^+ (q)~=~0$ and $S^- (q+1) = M(q+1) = K(q)$. Then $a = b = 1$, and condition (1) holds.

Let $n > 1$. Let $C$ be a vertex of degree $3$, $C \neq A$, $C \neq A_q$, and suppose that $f$ is the first edge of the path $[C,q+1]$ which is adjacent to the path $[A,q]$. Without loss of generality we can assume, by Remark 2.1, that $f$ is a left branch of $[A,q]$. Hence, $[A, q](f) =  S^+ (q)$.
Suppose that $\hat{e}_1 $, $\hat{e}_2 $, \ldots, $\hat{e}_n$ is a sequence of all consecutive edges of the path $[A, q+1]$
  which are adjacent to the path $[C, q]$ in vertices $\hat{E}_1$, $\hat{E}_2$, \ldots, $\hat{E}_n$, respectively. Note that  $E_j = E_{j+1}$ for $j$ even, $\hat{E}_j = \hat{E}_{j+1}$ for $j$ odd, and a segment $E_j \hat{E}_j$
  of the path $[A,q+1]$ has the length $|E_j \hat{E}_j | = K(q)$. Hence, segments $A E_b$ and $A \hat{E}_b$ of the path  $[A,q+1]$ have the lengths:
\\[6pt]
 (i) $\left\{
 \begin{array}{ll}
 |A E_b | = b  K(q),  & \mbox{for $b$ even} ,
   \\[4pt]
 |A \hat{E}_b | =  b  K(q),  & \mbox{for $b$ odd}.
   \end{array}
   \right.
 $
\\[6pt]
By Remark ~2.1 and Lemma~2.1(1), we have
$$
\begin{array}{l}
[A_q, q] (e_j ) - [A_q, q] (e_i)
\equiv [A, q] (e_j ) - [A, q] (e_i) \equiv [C, q](\hat{e}_i) - [C, q](\hat{e}_j )
\\[4pt]
\qquad\equiv
[C_q, q](\hat{e}_i) - [C_q, q](\hat{e}_j ) \pmod{ 2M(q)},
 \hbox{ for } \ 1 \leqslant i,  j \leqslant n.
\end{array}
$$
From Theorem~3.1(2) it follows that $n = \afrac{M(q)}{d}$. Hence, by Lemma~3.1(1), we obtain
\\[6pt]
(ii) \quad $\begin{array}{l}
[A_q, q] (e_j ) - [A_q, q] (e_i) \equiv [C, q](\hat{e}_i) - [C, q](\hat{e}_j )
\\[4pt]
\qquad\equiv
[C_q, q](\hat{e}_i) - [C_q, q](\hat{e}_j ) \equiv 0  \pmod{2d},
 \hbox{ for } \  1 \leqslant i,  j \leqslant n.
\end{array}
$
\\[6pt]
By Lemma 3.1(3) we get
 \\[6pt]
$[A, q](e_b) = \left\{\begin{array}{lll}
S^+(q) + d, & \hbox{for} \  a \ \hbox{even}  ,
\\[4pt]
M(q) - d, & \hbox{for} \ b \ \hbox{even}  ,
\\[4pt]
S^+(q) + M(q) + d, & \hbox{for} \  a \ \hbox{and} \ b \neq 1 \ \hbox{both odd} .
\end{array}\right.
$
\\[6pt]
 Since $[A, q](f) =  S^+ (q)$,  Lemma~2.1(1) shows that
$$
[C,q] (\hat{e_b}) \equiv [A, q ] (f) - [A, q ] (e_b) \equiv S^+ (q) - [A, q ] (e_b) \pmod {2M(q)}.
$$
Accordingly, by Remark~2.1, we obtain
 \\[6pt]
  (iii)\quad
 $\left\{
 \begin{array}{ll}
 [C, q] (\hat{e}_b)  =  2M(q) - d, &  \hbox{for} \ a \ \hbox{even},
 \\[4pt]
 [A_q,q](e_b) =   2M(q) - d, & \hbox{for} \  b \ \hbox{even},
 \\[4pt]
 [C_q , q](\hat{e}_b)  =  2M(q) - d, &  \hbox{for} \ a \ \hbox{and} \ b \neq 1 \ \hbox{both odd}.
 \end{array}
 \right.$
 \\[6pt]
Hence, for $b$ even ($b \neq 1$ odd) $e_b$ ($\hat{e}_b$, respectively) is a right branch of the path $[A_q, q]$ ($[C,q]$, or $[C_q ,q]$, respectively). For $b$ even ($b \neq 1$ odd), suppose that $g$ is the first edge of the directed path $[A_q, q]$ ($[C,q]$, or $[C_q ,q]$) which is adjacent to the directed path $[A, q+1]$.  By (ii)--(iii), $E_b$ ($\hat{E}_b$, respectively) is the common head of the arcs $g$ and $e_b$ ($\hat{e}_b$, respectively). Hence, $g$ is a left branch of the path $[A, q+1]$. Thus, by~(i), $S^- (q+1) = [A, q+1](g) = b K(q)$, and condition~(1) holds.

Condition (2) follows from (1) and Theorem~2.1.
\qed
\end{proof}

\begin{example}\label{exam3.1}
Let $\{ a_j \}$ be the Fibonacci sequence:
$$
a_1 = a_2 = 1 \quad \hbox{and} \quad a_{j+2} = a_j + a_{j+1}  \quad
 \hbox{for} \ \ j \in N .
$$
We will check that
$$
\{(1, a_{2n+1}a_{2n+2}, a_{2n}a_{2n+2}), (a_{2n+2}, a_{2n+1}, 0), (a_{2n+1}, a_{2n+2}, a_{2n})\}
$$
is the orbit of a graph in $\cal{P}$. Notice that for $n=1$ we obtain the orbit
$$\{(1,6,3), (3,2,0), (2,3,1)\}.$$
\end{example}
\begin{proof}
Since $a_j/a_{j+1}$ is the $j$-th convergence to $(\sqrt{5}-1)/{2}$, $j \in \mathbb{N}$, we have the following conditions (see Schmidt \cite[Lemma~3C,~3D]{flo9}):
\\[6pt]
(1) $a^2_{j+1} - a_j a_{j+2} = (-1)^j$,
\\[6pt]
(2) $a_{j+3} a_j - a_{j+2} a_{j+1} = (-1)^{j+1}$.
\\[6pt]
If $(K(1), M(1), S^+(1)) = (1, a_{2n+1} a_{2n+2}, a_{2n} a_{2n+2})$ then, by (1),
$$
a_{2n-1} M(1) - a_{2n} S^+ (1) = a_{2n+2}.
$$
 Hence, by Theorem~3.1(3-4) we have
$$
K(2) = a_{2n+2},\ M(2) = a_{2n+1},
$$
and, by Theorem~3.2(2),
$$
S^+(2) \equiv a_{2n} K(1) - K(2) = a_{2n} - a_{2n+2} = -a_{2n+1} \equiv 0 \pmod{  a_{2n+1}}.
$$
If $(K(2), M(2), S^+(2)) = (a_{2n+2}, a_{2n+1}, 0)$, then $M(2) - S^+ (2) = a_{2n+1}$. Hence, by Theorem~3.1(3--4) we obtain
$$
K(3) = a_{2n+1},\ M(3) = a_{2n+2},
$$
and, by  Theorem~3.2(2),
$$
S^+(3) \equiv K(2) - K(3) =  a_{2n+2} - a_{2n+1} = a_{2n} \pmod{ a_{2n+2}}.
$$
If $(K(3), M(3), S^+ (3)) = (a_{2n+1}, a_{2n+2}, a_{2n})$, then, by (2),
$$
a_{2n-1} M(3) - a_{2n+1} S^+ (3) = 1.
$$
Hence, by Theorem~3.1(3--4) we have
$$
K(1) = 1,\ M(1) = a_{2n+1} a_{2n+2},
$$
and, by Theorem~3.2(2) and (1),
$$
S^+(1) \equiv a_{2n+1}  K(3) - K(1) =  a^2_{2n+1} -1 = a_{2n} a_{2n+2} \pmod{ a_{2n+1} a_{2n+2}}.
$$
\qed
\end{proof}

\section{One point orbits of graphs in $\cal{P}$}
We recall that  $X = \{ (k,m,s) \in \mathbb{\mathbb{Z}}^3 :  1\leqslant k\, , 1 \leqslant m\, , 0 \leqslant s < m \}$ is the set of all index-vectors of graphs in $\cal{P}$. In the following theorem we characterize one point orbits.
\begin{theorem}\label{theo4.1}
$\{(k,m,s) \}\in X$ is a one point orbit of a graph in $\cal{P}$ if and only if $ m = kn$, $s = kx$, where $0 \leqslant x < n$ are integers such that $n$ is a divisor of $x^2 + x + 1$.
\end{theorem}
\begin{proof}
Let $(k,m,s)$ be an index-vector of a graph in $\cal{P}$. It is easy to prove that the following conditions are equivalent (equivalence (ii)--(iii) follows from Theorem 3.1(3--4) and Theorem 3.2(2)):
\\[6pt]
(i) $\{(k,m,s) \}$ is a one point orbit,
\\[6pt]
(ii) $(k,m,s) = (K(q), M(q), S^+ (q)) = (K(q+1), M(q+1), S^+ (q+1))$,
\\[6pt]
(iii) $k = |s, m|$ and $s = bk - k$, where $b$ is an integer such that $0 < b \leqslant \afrac{m}{k}$ and  $bs \equiv -k  \pmod{m}$,
\\[6pt]
(iv) $m = kn$,  $s = kx = bk - k$, where $n \geqslant 1$, $x \geqslant 0$ and $0 < b \leqslant n$ are integers such that  $bx \equiv -1  \pmod{n}$,
\\[6pt]
(v)  $m=kn$,  $s=kx$, where $0\leqslant x < n$ are integers such that
 $n$ is a divisor of $x^2 + x + 1$.
\qed
\end{proof}
\begin{remark}\label{rem4.1}
Notice that if $\{(k,m,s)\}$ is a one point orbit of a graph $G\in{\cal P}$, then, by Theorem \1{2.1},
 $\{(k,m,m-s-k)\}$ is the one point orbit of the mirror reflection of $G$. Hence, by Theorem \1{4.1}, $n$ is a divisor of $x^2 + x + 1$ if and only if $n$ is a divisor of $(n-x-1)^2 + (n-x-1) + 1$, which is confirmed by the following equivalence
$$
x^2 + x + 1 = an \ \Leftrightarrow \ (n-x-1)^2 + (n-x-1) + 1 = (n-2x-1+a)n.
$$
 \end{remark}

\begin{example}\label{exam4.1}
Notice that $(a,n,x) = (1,1,0), (1,3,1), (1,7,2)$ and $(1,13,3)$ are all integral solutions of the diophantine equation $$
x^2 + x + 1 = a\cdot{n},
\quad \mbox{for \/ $0\leqslant x\leqslant3$ \/ and \/  $x < n$}.
$$
Hence, by Theorem \1{4.1},
$\{(k,k,0)\}$, for $k \in \mathbb{N}$, $\{(1,3,1)\}$, $\{(1,7,2)\}$ and $\{(1,13,3)\}$ are all one point orbits with $s \leqslant 3$. Notice that $K^4$ (tetrahedron) has the one point orbit $\{(1,1,0)\}$. Let $G_0$, $G_1$, $G_2$ and $G_3$ be graphs in $\cal P$ with one point orbits
$$
 \{(4,4,0)\}, \{(1,3,1)\}, \{(1,7,2)\}
\mbox{ and }\{(1,13,3)\},
$$
 respectively. Let us consider a solid regular tetrahedron with closed $3$-faces $f_1$, $f_2$, $f_4$, $f_4$. It is easy to check that $G_j$, $j= 0, 1, 2, 3$, can be embedded on the sphere of the solid regular tetrahedron in such a way that all four induced plane graphs
$G_j[{V_j}\cap{f_1}]$, \ldots, $G_j[{V_j}\cap{f_4}]$ are op-equivalent
 to the plane graph $Q_j$ shown in\/ \1{Fig.~3}.

We \textbf{conjecture} that each graph $G\in{\cal P}$ with one point orbit, and the vertex set $V$, can be embedded on the sphere of the solid regular tetrahedron in such a way that all four induced plane graphs $G[{V}\cap{f_1}]$, \ldots, $G[{V}\cap{f_4}]$ are op-equivalent.
\begin{figure}
\centering
\begin{tabular}{c}
\begin{tikzpicture}[scale=1]

\coordinate  (A) at (0,0);
\coordinate  (B) at (4,0);
\draw (A) -- (B);
\coordinate (D) at
($ (A) ! .5 ! (B) ! {sin(60)*2} ! 90:(B) $) {};
\draw (A) -- (D) -- (B);

\draw ($(A)!0.25!(D)$) -- ($(B)!0.25!(D)$);
\draw ($(A)!0.5!(D)$) -- ($(B)!0.5!(D)$);
\draw ($(A)!0.75!(D)$) -- ($(B)!0.75!(D)$);

\draw ($(A)!0.25!(D)$) -- ($(A)!0.25!(B)$);
\draw ($(A)!0.5!(D)$) -- ($(A)!0.5!(B)$);
\draw ($(A)!0.75!(D)$) -- ($(A)!0.75!(B)$);

\draw ($(B)!0.75!(D)$) -- ($(A)!0.25!(B)$);
\draw ($(B)!0.5!(D)$) -- ($(A)!0.5!(B)$);
\draw ($(B)!0.25!(D)$) -- ($(A)!0.75!(B)$);


\end{tikzpicture}\\
graph $Q_0$
\end{tabular}
\\[12pt]
\centering
\begin{tabular}{ccc}
\begin{tikzpicture}[scale=1]

\coordinate  (A) at (0,0);
\coordinate  (B) at (4,0);
\draw (A) -- (B);
\coordinate (D) at
($ (A) ! .5 ! (B) ! {sin(60)*2} ! 90:(B) $) {};
\draw (A) -- (D) -- (B);
\coordinate (E) at (intersection of 0,0--30:2
and 2,0--2,1); 
\draw (A)--(E);
\draw (B)--(E);
\draw (D)--(E);

\end{tikzpicture}&

%

\begin{tikzpicture}[scale=1.368] 

\coordinate  (A) at (0,0);
\coordinate  (B) at (3,0);
\draw (A) -- (B);
\coordinate (D) at
($ (A) ! .5 ! (B) ! {sin(60)*2} ! 90:(B) $) {};
\draw (A) -- (D) -- (B);

\draw ($(A)!0.6667!(D)$) -- ($(B)!0.6667!(D)$);
\draw ($(D)!0.6667!(A)$) -- ($(B)!0.6667!(A)$);
\draw ($(A)!0.6667!(B)$) -- ($(D)!0.6667!(B)$);

\draw ($(A)!0.6667!(B)$) -- ($(B)!0.6667!(D)$);
\draw ($(B)!0.6667!(D)$) -- ($(D)!0.6667!(A)$);
\draw ($(D)!0.6667!(A)$) -- ($(A)!0.6667!(B)$);

\end{tikzpicture}
&

%

\begin{tikzpicture}[scale=1.1]

\coordinate  (A) at (0,0);
\coordinate  (B) at (1,0);
\coordinate  (C) at ($(B)+ (30:1cm)$);
\coordinate  (D) at ($(C)+ (-30:1cm)$);
\coordinate  (E) at ($(D)+ (0:1cm)$);
\coordinate  (F) at ($(E)+ (120:1cm)$);
\coordinate  (G) at ($(F)+ (150:1cm)$);
\coordinate  (H) at ($(G)+ (90:1cm)$);
\coordinate  (I) at ($(H)+ (120:1cm)$);
\coordinate  (J) at ($(I)+ (-120:1cm)$);
\coordinate  (K) at ($(J)+ (-90:1cm)$);
\coordinate  (L) at ($(K)+ (-150:1cm)$);

\coordinate  (M) at (intersection of K--E and A--G);

\draw (B)--(C)--(D)--(F)--(G)--(H)--(J)--(K)--(L)--cycle;
\draw (B)--(D);
\draw (F)--(H);
\draw (J)--(L);
\draw (C)--(G)--(K)--cycle;

\draw (L)--(C);
\draw (D)--(G);
\draw (H)--(K);

\draw (M)--(C);
\draw (M)--(G);
\draw (M)--(K);

\end{tikzpicture}\\
graph $Q_1$ & graph $Q_2$ & graph $Q_3$
\end{tabular}
\caption{}
\end{figure}

\end{example}

\begin{theorem}\label{theo4.2}
$G \in {\cal P}$ is double mirror symmetric if and only if $G$ has a~one point orbit of the form $\{(k,k,0)\}$ or $\{(k,3k,k)\}$ for some $k\in \mathbb{N}$.
\end{theorem}

\begin{proof}
Let $G \in {\cal P}$ and suppose that $\{(K(q),M(q),S^+ (q)): q\in Q\}$ is the orbit of $G$. First we prove that if
$S^+ (q) + S^- (q) = M(q)$,  for  $q = 1, 2$, then $G$ has a one point orbit of the form $\{(k,2s+k,s)\}$.
If $S^+(q) + S^-(q) = M(q)$, for  $q = 1, 2$, then by Theorem~3.1(3) and Corollary~3.1(3) we conclude that $K(0) = K(1) = K(2)= k$. Hence,
${M(0) = M(1) = M(2)= m}$, by Theorem~3.1(4). Suppose that $a_q$, $b_q$, for $q\in Q$, are integers such that $a_q m - b_qS^+ (q) = k$ and $1 \leqslant b_q \leqslant \afrac{m}{k}$. By Theorem~3.2(1--2), we deduce that  $S^+ (q+1) = b_q k - k$ and ${S^- (q+1) = b_q k}$. Since $S^+ (q+1) + S^- (q+1) = m$ for  $q = 0$, $1$, we see that $b_0 = b_1$, $S^+ (1) = S^+ (2) = s$, and ${s + (s+k) =  m}$. Since $(K(1), M(q), S^+ (1)) = (K(2), M(2), S^+ (2)) = (k,2s+k,s)$, we have $(K(0), M(0), S^+ (0)) = (k,2s+k,s)$. This completes the proof of the implication. The opposite implication follows from Theorem~2.1.

It is easy to see that the following conditions are equivalent (the equivalence
(i)--(ii) follows from Theorem~4.1):
\\[6pt]
(i) $\{ (k,2s+k,s) \}$ is a one point orbit of $G$,
 \\[6pt]
(ii) $m = 2s + k = kn$,  $s = kx$, where  $ 0 \leqslant x < n$  are integers such that $n$ is a divisor of $x^2 + x + 1$,
\\[6pt]
(iii)
$m = k(2x+1)$,  $s = kx$,  where  integers $x  \geqslant 0$  and $a > 0$ are solutions of the equation  $x^2 + x + 1 = a(2x+1)$.
\\[6pt]
Let $D$ be the determinant of the quadratic equation $x^2 + x(1 - 2a) + 1 - a = 0$. Since $D = 4a^2 - 3$
 is a square of an integer, it follows that $a = 1$. Hence, $x = 0$ or $x = 1$, which completes the proof.
 \qed
\end{proof}

\section{Hamilton bonds with the end-trees of the same order}
Let $G$ be a $2$-connected plane triangulation, and suppose that $S$ and $T$ are the end-trees of a Hamilton bond in $G$. Let us denote by $f'_i$ ($f''_i$) the number of vertices of degree $i$ contained in $S$ ($T$, respectively). Tutte \cite{flo11} proved the following identity, which is the dual version of the well-known Grinberg's theorem \cite{flo5}:
\begin{equation}
\sum_{i}(i-2)f'_i=
\sum_{i}(i-2)f''_i.
\label{eq:r1}
\end{equation}
Let us denote by $f_i$ the number of vertices of degree $i$ of the graph $G$. Euler's equation becomes:
\begin{equation}
\sum_{i}(6-i)f_i=12.
\label{eq:r2}
\end{equation}
Recall that $\cal{P}$ ($\cal{H}$) is the family of all $2$-connected plane triangulations all of whose vertices are of degree $3$ or $6$ (at most~$6$, respectively).
\begin{theorem}\label{theo5.1}
If $G\in \cal{P}$, then for every Hamilton bond, the end-trees of the bond have the same number of vertices of degree $6$, and the same number of vertices of degree $3$ in $G$.
\end{theorem}
\begin{proof}
Let $S$ and $T$ be the end-trees of a Hamilton bond in $G$.
By equality (\ref{eq:r1}) we have
$4f'_6+f'_3=4f''_6+f''_3$.
Hence,
$f'_3\equiv f''_3 \pmod{4}$.
Because of $f'_3+f''_3=4$
we have two cases:
$f'_3=4$
or
$f'_3=2=f''_3$.
In the first case we have
$4f'_6+4=4f''_6$.
Accordingly, the number $f_6$ is odd. Hence, we have a~contradiction, because  the order of $G$ is even. In the second case we have
$f'_3=2=f''_3$
and we obtain
$f'_6=f''_6$.
\qed
\end{proof}
\begin{theorem}\label{theo5.2}
If $G\in \cal{H}$, then for every Hamilton bond in $G$, the difference in the orders of the end-trees of the bond is not greater then $3$.
\end{theorem}
\begin{proof}
Let $S$ and $T$ be the end-trees of a Hamilton bond in $G$. By equality (\ref{eq:r1}) and (\ref{eq:r2}) we obtain
\begin{eqnarray*}
\Bigl|
\sum_{i=3}^{6}f''_i-\sum_{i=3}^{6}f'_i
\Bigr|
&=&
\Bigl|
\sum_{i=3}^{5}(f''_i-f'_i)-\sum_{i=3}^{5}\frac{i-2}{4}(f''_i-f'_i)
\Bigr|\\[12pt]
&=&
\Bigl|
\frac14\sum_{i=3}^{5}(6-i)(f''_i-f'_i)
\Bigr|
\leqslant
\frac14\sum_{i=3}^{5}(6-i)f_i=3,
\end{eqnarray*}
which completes the proof.
\qed
\end{proof}

\section{Induced Caterpillars}
Let $\cal{P}$ be the family of all $2$-connected plane triangulations all of whose vertices are of degree~$3$ or~$6$. We recall that a graph $P \in \cal{P}$ is factorable into factors $P_0$, $P_1$, $P_2$ (indexed by elements of the cyclic group $Q = \{0,1,2\}$) satisfying the condition (*). Notice that $P$ has a Hamilton bond if and only if there exists a partitioning of the vertex set of $P$ into two subsets so that each induces a tree.  We show an example that such trees are not always equitable $2$-colorable.
\begin{example}\label{exam6.1}
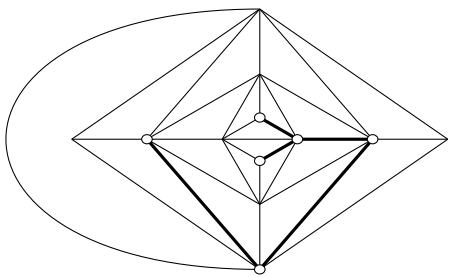
\begin{figure}
\label{Fig4}
\centering
\begin{tikzpicture}[yscale=0.866] 

\foreach \i in {0,...,5} \coordinate (s\i) at (\i,0); 

\coordinate (cg) at (2.5,0.333); 

\coordinate (cg1) at (2.5,1);

\coordinate (cg2) at (2.5,2);

\coordinate (cd) at (2.5,-0.333);

\coordinate (cd1) at (2.5,-1);

\coordinate (cd2) at (2.5,-2);

\draw (s0) -- (s5);
\draw (cg) -- (cg2);
\draw (cg) -- (s2);
\draw (cg) -- (s3);
\draw (cg1) -- (s2);
\draw (cg1) -- (s3);
\draw (cg1) -- (s1);
\draw (cg1) -- (s4);
\draw (cg2) -- (s1);
\draw (cg2) -- (s4);
\draw (cg2) -- (s0);
\draw (cg2) -- (s5);
\draw (cd) -- (cd2);
\draw (cd) -- (s2);
\draw (cd) -- (s3);
\draw (cd1) -- (s2);
\draw (cd1) -- (s3);
\draw (cd1) -- (s1);
\draw (cd1) -- (s4);
\draw (cd2) -- (s1);
\draw (cd2) -- (s4);
\draw (cd2) -- (s0);
\draw (cd2) -- (s5);
\draw (cg2) .. controls (-2,2) and (-2,-2) .. (cd2);
\draw[very thick] (s1) -- (cd2) -- (s4) --  (s3) -- (cg);
\draw[very thick] (s3) -- (cd);
\foreach \i in {cg,cd,s1,s3,s4,cd2} \filldraw[fill=white] (\i) circle (2pt);
\end{tikzpicture}
\caption{An induced tree (in bold) is not equitable $2$-colorable}
\end{figure}
Let $G \in \cal{P}$ be the graph of Figure $4$. Notice that $G$  contains two disjoint and induced trees whose vertices together span all of~$G$. However, the induced trees are not equitable $2$-colorable.
\end{example}
A \textsl{$k$-caterpillar}, $k \geq 1$, is a tree $T$ which contains a path $T_0$ such that $T-V(T_0)$ is a family of independent paths of the same order $k$. The path $T_0$ is referred to as the \textsl{spine} of $T$ (see Chartrand and Lesniak \cite{flo1}). Paths and $k$-caterpillars, for $k$ even, are called \textsl{even caterpillars}. Notice that even caterpillars are equitable $2$-colorable.

Let $H\in \cal{P}$ with $K(q) = 2$, for some $q \in Q$, where $(K(q), M(q), S^+(q))$ is a $q$-index-vector of the graph $H$. Goodey \cite{flo4} constructed a Hamiltonian cycle in every  $2$-connected cubic plane graph whose faces are only triangles or hexagons. In Lemma 6.1 we use a dual version of the Goodey's construction to partition the vertex set of $H$ into two subsets so that each induces an even caterpillar.
\begin{lemma}\label{emmao6.1}
The graph $H$ contains two disjoint and induced even caterpillars $T$ and $S$ \1{(}$T$ is a $(2d-2)$-caterpillar, where $d = |S^+(q) + 1, M(q)|$, and  $S$ is a path\1{)} whose vertices together span all of $H$. Moreover, if $\gamma_1$ is the cycle of class  $q$ in $H$, then
\begin{itemize}
\item [\1{(}\1{1}\1{)}]
$T \cap \gamma_1$ is a family of independent paths in $H$ with the same order $2d-1$, and $S \cap \gamma_1$ is an independent set of vertices.
\end{itemize}
\end{lemma}
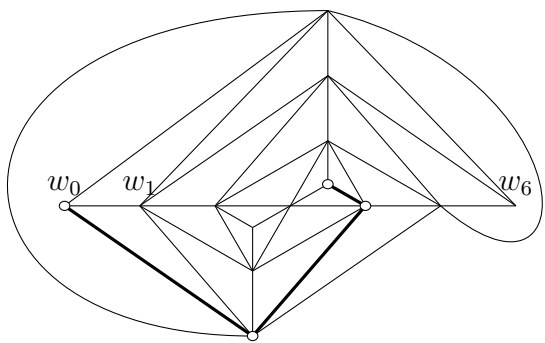
\begin{figure}
\label{Fig5}
\centering
\begin{tikzpicture}[yscale=0.866] 

\foreach \i in {0,...,6} \coordinate (s\i) at (\i,0); 

\coordinate (cg) at (3.5,0.333); 

\coordinate (cg1) at (3.5,1);

\coordinate (cg2) at (3.5,2);

\coordinate (cg3) at (3.5,3);

\coordinate (cd) at (2.5,-0.333);

\coordinate (cd1) at (2.5,-1);

\coordinate (cd2) at (2.5,-2);

\draw (s0) -- (s6);

\draw (cg) -- (cg3);
\draw (cg) -- (s3);
\draw (cg) -- (s4);
\draw (cg1) -- (s3);
\draw (cg1) -- (s4);
\draw (cg1) -- (s2);
\draw (cg1) -- (s5);
\draw (cg2) -- (s2);
\draw (cg2) -- (s5);
\draw (cg2) -- (s1);
\draw (cg2) -- (s6);
\draw (cg3) -- (s1);
\draw (cg3) -- (s6);
\draw (cg3) -- (s0);

\draw (cd) -- (cd2);
\draw (cd) -- (s2);
\draw (cd) -- (s3);
\draw (cd1) -- (s2);
\draw (cd1) -- (s3);
\draw (cd1) -- (s1);
\draw (cd1) -- (s4);
\draw (cd2) -- (s1);
\draw (cd2) -- (s4);
\draw (cd2) -- (s0);
\draw (cd2) -- (s5);

\draw (cg3) .. controls (-2,3) and (-2,-2) .. (cd2);

\draw (cg3) .. controls (7,2) and (7,-2) .. (s5);

\draw[very thick] (s0) -- (cd2) -- (s4) -- (cg);

\foreach \i in {cg,s0,s4,cd2} \filldraw[fill=white] (\i) circle (2pt);

\draw (s0) node[anchor=south]{$w_0$};
\draw (s1) node[anchor=south]{$w_1$};
\draw (s6) node[anchor=south]{$w_6$};

\end{tikzpicture}
\caption{A $[w_0, q-1]$ path (in bold) in the graph $H_\gamma$}
\end{figure}
\begin{proof}
Let $\gamma = v_0v_1 \ldots v_{M(q)}$ and $\gamma'$ be two maximal paths of class $q$, and suppose that $\gamma_1 = t_0t_1 \ldots t_{2M(q)-1}$ is the clock-wise oriented cycle of class $q$ in $H$. Without loss of generality we can assume that vertices $t_0$, $t_1$ are adjacent to $v_1$ (see Fig.~6). If $S^+(q) = M(q)-1$, then there exists a vertex $u \neq v_0$ of degree $3$ which is adjacent to $t_{2M(q)-1}$ and $t_0$. Then the set
$W = \{u, v_0, t_0, t_1, \ldots, t_{2M(q)-2}\}$ ($V(H) - W$) induces a $(2M(q)-2)$-caterpillar $T$ (a path $S$, respectively) satisfying condition (1).

Let $S^+(q) < M(q)-1$. In the graph $H - V(\gamma)$ we identify successive vertices and edges of the path $t_0t_1 \ldots t_{M(q)}$   with successive vertices and edges of the path $t_0t_{2M(q)-1} t_{2M(q)-2} \ldots t_{M(q)}$. After the identification we obtain a path $\delta = w_0w_1 \ldots w_{M(q)}$ and a graph $H_{\gamma} \in \cal{P}$ (see Fig.~5). Notice that $\delta$ and $\gamma'$ are two maximal paths of the same class (say $q$) in  $H_{\gamma}$. Hence,  $(K_{\gamma}(q), M_{\gamma}(q), S^+_{\gamma}(q)) = (1, M(q), S^+ (q))$ is the $q$-index-vector of the graph $H_\gamma$. Let $e_1, e_2, \ldots, e_n$ be a sequence of all consecutive edges of the path $[w_0, q-1]$ which are adjacent to the path $\delta$ (see Fig.~5). Since $S^-_{\gamma}(q) = S^+_{\gamma}(q)+1 = S^+(q) +1 < M(q)$, we have $n > 1$. By Lemma 3.1(1) and Corollary 3.1(2), we obtain
\begin{itemize}
\item [(2)]
$\{[w_0,q](e_1),[w_0,q](e_2),\ldots, [w_0,q](e_n)\} = \{0, 2d, 4d, \ldots, 2M(q)-2d\}$,
\end{itemize}
where $d = |S^+(q) + 1, M(q)|$. Let $I = \{0 \leqslant i \leqslant  M(q): w_i \in  V([w_0, q-1])\}$.  We can consider $V_0 = V([w_0, q-1])\cap V(\gamma')$ as a set of vertices in $H$. It is not difficult to see that the following set
$$V_1 = V_0 \cup \bigcup_{i \in I} \{v_{i}\} \cup \bigcup_{i \in I} \{t_{i}\}\cup \bigcup_{i \in I \backslash \{0, M(q)\}} \{t_{2M(q)-i}\}$$
induces a path $T_0$ in $H$ (see Fig.~6). Accordingly, by (2), the following set
\begin{align*}
V_2 = V_1 & \cup \bigcup _{i \in I} \{t_{i+1}, t_{i+2}, \ldots, t_{i+2d-2}\}
\cr & \cup \bigcup _{i \in I \backslash \{0, M(q)\}}\{t_{2M(q)-i+1}, t_{2M(q)-i+2},\ldots, t_{2M(q)-i+2d-2}\}
\end{align*}
induces a $(2d-2)$-caterpillar $T$ in $H$ with the spine $T_0$ (see Fig.~6). Notice that $V(H)-V_2$ induces a path $S$ in $H$, and condition (1) is satisfied.
\qed
\end{proof}

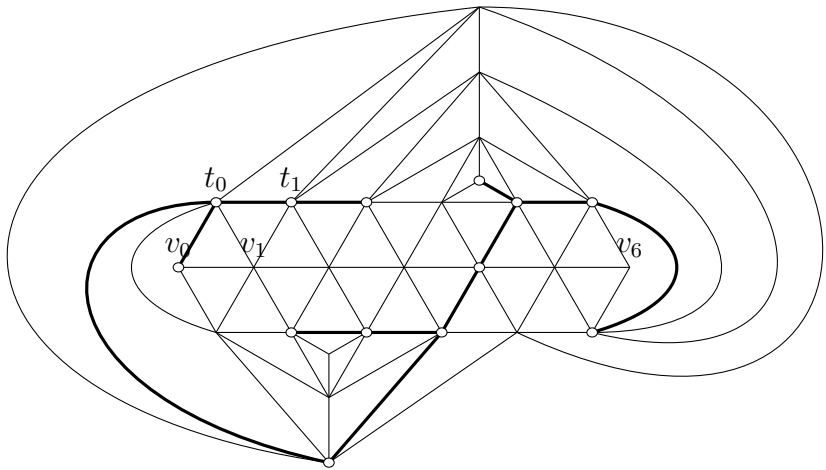
\begin{figure}
\label{Fig6}
\centering
\begin{tikzpicture}[yscale=0.866] 

\foreach \i in {0,...,6} \coordinate (s\i) at (\i,0); 

\foreach \j in {0,...,5} \coordinate (g\j) at (\j+0.5,1); 

\foreach \k in {0,...,5} \coordinate (d\k) at (\k+0.5,-1); 

\coordinate (cg) at (4,1.333); 

\coordinate (cg1) at (4,2);

\coordinate (cg2) at (4,3);

\coordinate (cg3) at (4,4);

\coordinate (cd) at (2,-1.333);

\coordinate (cd1) at (2,-2);

\coordinate (cd2) at (2,-3);

\coordinate (cl) at (-1,0);
\coordinate (cp) at (7,0); 

\draw (s0) -- (s6);
\draw (g0) -- (g5);
\draw (d0) -- (d5);

\draw (cg) -- (cg3);
\draw (cg) -- (g3);
\draw (cg) -- (g4);
\draw (cg1) -- (g3);
\draw (cg1) -- (g4);
\draw (cg1) -- (g2);
\draw (cg1) -- (g5);
\draw (cg2) -- (g2);
\draw (cg2) -- (g5);
\draw (cg2) -- (g1);
\draw (cg3) -- (g1);
\draw (cg3) -- (g0);

\draw (cd) -- (cd2);
\draw (cd) -- (d1);
\draw (cd) -- (d2);
\draw (cd1) -- (d1);
\draw (cd1) -- (d2);
\draw (cd1) -- (d0);
\draw (cd1) -- (d3);
\draw (cd2) -- (d0);
\draw (cd2) -- (d3);
\draw (cd2) -- (d4);

\draw (s0) -- (g0);
\draw (d0) -- (g1);
\draw (d1) -- (g2);
\draw (d2) -- (g3);
\draw (d3) -- (g4);
\draw (d4) -- (g5);
\draw (d5) -- (s6);

\draw (s0) -- (d0);
\draw (g0) -- (d1);
\draw (g1) -- (d2);
\draw (g2) -- (d3);
\draw (g3) -- (d4);
\draw (g4) -- (d5);
\draw (g5) -- (s6);

\coordinate (tl) at (-0.5,0); 
\coordinate (ll) at (-4,0);

\draw (g0) .. controls (-1,0.5) and (-1,-0.5) .. (d0);
\draw (g5) .. controls (7,0.5) and (7,-0.5) .. (d5);

\draw (g0) .. controls (-2,1) and (-2,-2) .. (cd2);
\draw (cg2) .. controls (8,1) and (8,-1) .. (d5);

\draw (cg3) .. controls (-4,3) and (-4,-2) .. (cd2);

\draw (cg3) .. controls (9,2) and (9,-2) .. (d5);

\draw (cg3) .. controls (10,4) and (10,-4) .. (d4);

\draw[very thick] (s0) -- (g0) -- (g2);
\draw[very thick]  (g0) .. controls (-2,1) and (-2,-2) .. (cd2);
\draw[very thick] (cd2) -- (d3);
\draw[very thick] (d3) -- (d1);
\draw[very thick] (d3) -- (g4);
\draw[very thick] (g4) -- (cg);
\draw[very thick] (g4) -- (g5);
\draw[very thick]  (g5) .. controls (7,0.5) and (7,-0.5) .. (d5);

\foreach \i in {cg,g0,g1,g2,g4,g5,
s0,s4,d1,d2,d3,d5,cd2} \filldraw[fill=white] (\i) circle (2pt);

\draw (s0) node[anchor=south]{$v_0$};
\draw (s1) node[anchor=south]{$v_1$};
\draw (g0) node[anchor=south]{$t_0$};
\draw (g1) node[anchor=south]{$t_1$};
\draw (s6) node[anchor=south]{$v_6$};

\end{tikzpicture}
\caption{An even caterpillar $T$ (in bold) in the graph $H$}
\end{figure}
In Theorem 6.1 we prove that if $P \in \cal{P}$ has the order $2n$ and $n$ is odd, then there is possible to partition the vertex set of $P$ into two subsets so that each induces an even caterpillar. Hence, by Theorem 5.1, the two even caterpillars are equitable $2$-colorable and have the same order.
In Theorem~6.2 we prove that if $P$ has an index-vector $(K(q), M(q), S^+(q))$ such that $M(q)$ is odd and $K(q) \geqslant \frac{M(q)}{3}$, then there is a partitioning of the vertex set of $P$ into two subsets so that each induces a path.
\begin{theorem}\label{theo6.1}
Let $P \in \cal{P}$. If $P$ has the order $2n$ and $n$ is odd, then $P$ contains two disjoint and induced even caterpillars which vertices together span all of~$P$.
\end{theorem}
\begin{proof}
Let $P \in \cal{P}$ has the order $2n$, and suppose that $n$ is odd. Let $(K(q), M(q), S^+(q))$ be the $q$-index-vector of $P$, $q\in Q$. First we prove that $K(q)$ is even for some $q \in Q$. We know that  $2K(q)M(q)+2 = 2n$ for every $q\in Q$. Hence, if $K(q)$ is odd, then $M(q)$ is even, because  $K(q)M(q) = n-1$ is even. By Theorem 2.1, $S^-(q) -S^+(q)  \equiv K(q) \pmod {M(q)}$, whence $S^+(q)$ or $S^-(q)$ is even. By Theorem 3.1(3) and Corollary 3.1(3), $K(q\pm1) = |S^\pm (q), M(q)|$, whence  $K(q+1)$ or $K(q-1)$ is even.

Let now $K(q) = k$ be even, and suppose that $\gamma_0$, $\gamma'$ are maximal paths of class $q$, and  $\gamma_1 ,\gamma_2,\ldots,\gamma_{k-1}$  are clock-wise oriented cycles of class $q$ such that vertices of $\gamma_j$ are adjacent to vertices of $\gamma_{j-1}$, $1 \leqslant j < k$. We will prove that $P$ contains two disjoint and induced even caterpillars $T_k$ and $S_k$ which vertices together span all of $P$, and the following condition is satisfied
\begin{itemize}
\item [(3)]
$\{T_k \cap \gamma_{j}:  j \hbox{ odd}, 1 \leqslant j < k\} \cup \{S_k \cap \gamma_{j}:  j \hbox{ even}, 1 \leqslant j < k\}$
is a family of independent paths in $P$ with the same odd order, and  $\{T_k \cap \gamma_{j}:  j \hbox{ even}, 1 \leqslant j < k\} \cup \{S_k \cap \gamma_{j}:  j \hbox{ odd}, 1 \leqslant j < k\}$ is an independent set of vertices in $P$.
\end{itemize}
\begin{figure}
\label{Fig7}
\hspace*{-18mm}
\begin{tikzpicture}[yscale=0.866] 

\begin{scope}
\foreach \i in {0,...,6} \coordinate (s\i) at (\i,0); 

\foreach \j in {0,...,5} \coordinate (g\j) at (\j+0.5,1); 

\foreach \k in {0,...,5} \coordinate (d\k) at (\k+0.5,-1); 

\foreach \jj in {1,...,5} \coordinate (gg\jj) at (\jj,2); 

\foreach \kk in {1,...,5} \coordinate (dd\kk) at (\kk,-2); 

\foreach \jjj in {0,...,5} \coordinate (ggg\jjj) at (\jjj+0.5,3); 

\foreach \kkk in {0,...,5} \coordinate (ddd\kkk) at (\kkk+0.5,-3); 

\coordinate (cg) at (4,3.333); 

\coordinate (cg1) at (4,4);

\coordinate (cg2) at (4,5);

\coordinate (cg3) at (4,6);

\coordinate (cd) at (2,-3.333);

\coordinate (cd1) at (2,-4);

\coordinate (cd2) at (2,-5);

\coordinate (cl) at (-1,0);
\coordinate (cp) at (7,0); 

\draw (s0) -- (s6);
\draw (g0) -- (g5);
\draw (d0) -- (d5);
\draw (gg1) -- (gg5);
\draw (dd1) -- (dd5);
\draw (ggg0) -- (ggg5);
\draw (ddd0) -- (ddd5);

\draw (cg) -- (cg3);
\draw (cg) -- (ggg3);
\draw (cg) -- (ggg4);
\draw (cg1) -- (ggg3);
\draw (cg1) -- (ggg4);
\draw (cg1) -- (ggg2);
\draw (cg1) -- (ggg5);
\draw (cg2) -- (ggg2);
\draw (cg2) -- (ggg5);
\draw (cg2) -- (ggg1);
\draw (cg3) -- (ggg1);
\draw (cg3) -- (ggg0);

\draw (cd) -- (cd2);
\draw (cd) -- (ddd1);
\draw (cd) -- (ddd2);
\draw (cd1) -- (ddd1);
\draw (cd1) -- (ddd2);
\draw (cd1) -- (ddd0);
\draw (cd1) -- (ddd3);
\draw (cd2) -- (ddd0);
\draw (cd2) -- (ddd3);
\draw (cd2) -- (ddd4);

\draw (s0) -- (ggg1);
\draw (d0) -- (ggg2);
\draw (ddd0) -- (ggg3);
\draw (ddd1) -- (ggg4);
\draw (ddd2) -- (ggg5);
\draw (ddd3) -- (g5);
\draw (ddd4) -- (s6);

\draw (s0) -- (ddd1);
\draw (g0) -- (ddd2);
\draw (ggg0) -- (ddd3);
\draw (ggg1) -- (ddd4);
\draw (ggg2) -- (ddd5);
\draw (ggg3) -- (d5);
\draw (ggg4) -- (s6);

\coordinate (tl) at (-0.5,0); 
\coordinate (ll) at (-4,0);

\draw (g0) .. controls (-1,0.5) and (-1,-0.5) .. (d0);
\draw (g5) .. controls (7,0.5) and (7,-0.5) .. (d5);

\draw (g0) to[bend right] (cl);
\draw (cl) to[bend right] (d0);
\draw (g5) to[bend left] (cp);
\draw (cp) to[bend left] (d5);

\draw (gg1) to[bend right] (cl);
\draw (cl) to[bend right] (dd1);
\draw (gg5) to[bend left] (cp);
\draw (cp) to[bend left] (dd5);

\draw (ggg0) to[bend right] (cl);
\draw (cl) to[bend right] (ddd0);
\draw (ggg5) to[bend left] (cp);
\draw (cp) to[bend left] (ddd5);

\draw (ggg0) .. controls (-3,2) and (-3,-2) .. (ddd0);
\draw (ggg5) .. controls (9,2) and (9,-2) .. (ddd5);

\draw (ggg0) .. controls (-4,2) and (-4,-2) .. (cd2);

\draw (cg2) .. controls (10,3) and (10,-3) .. (ddd5);

\draw (cg3) .. controls (-5,6) and (-5,-5) .. (cd2);

\draw (cg3) .. controls (11,4) and (11,-3) .. (ddd5);
\draw (cg3) .. controls (12,5) and (12,-5) .. (ddd4);
\draw[very thick] (s0) -- (g0) -- (g2);
\draw[very thick] (g0) to[bend right] (cl);
\draw[very thick]  (ggg0) to[bend right] (cl);
\draw[very thick]  (ggg0) -- (ggg2);
\draw[very thick] (ggg0) .. controls (-4,2) and (-4,-2) .. (cd2);
\draw[very thick] (cd2) -- (ddd3);
\draw[very thick] (ddd3) -- (ddd1);
\draw[very thick] (ddd3) -- (dd4);
\draw[very thick] (dd4) -- (d3);
\draw[very thick] (d3) -- (d1);
\draw[very thick] (d3) -- (g4);
\draw[very thick] (g4) -- (g5);
\draw[very thick] (g5) .. controls (7,0.5) and (7,-0.5) .. (d5);
\draw[very thick] (g4) -- (gg4);
\draw[very thick] (gg4) -- (ggg4);
\draw[very thick] (ggg4) -- (cg);
\draw[very thick] (ggg4) -- (ggg5);
\draw[very thick] (ggg5) .. controls (9,2) and (9,-2) .. (ddd5);

\foreach \i in {cg,ggg0,ggg1,ggg2,ggg4,ggg5,gg4,g0,g1,g2,g4,g5,
cl,s0,s4,d1,d2,d3,d5,dd4,ddd1,ddd2,ddd3,ddd5,cd2} \filldraw[fill=white] (\i) circle (2pt);

\draw (s0) node[anchor=south]{$v_0$};
\draw (s1) node[anchor=south]{$v_1$};
\draw (g0) node[anchor=south]{$x_0$};
\draw (g1) node[anchor=south]{$x_1$};
\draw (s6) node[anchor=south]{$v_6$};
\draw (gg1) node[anchor=south]{$y_0$};
\draw (gg2) node[anchor=south]{$y_1$};
\draw (ggg1) node[anchor=south]{$z_1$};
\draw (ggg0) node[anchor=south]{$z_0$};

\end{scope}
\end{tikzpicture}
\caption{An even caterpillar $T$ (in bold) in the graph $P$}
\end{figure}
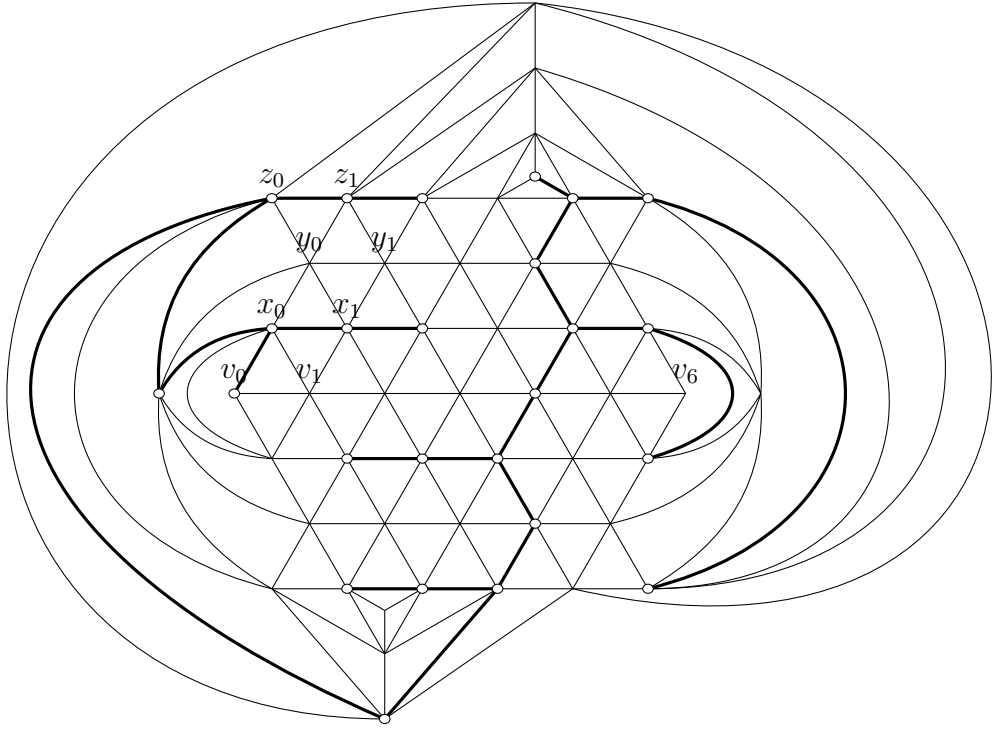
We proceed by induction on the even number $K(q) = k$. By Lemma 6.1, we can assume that $k \geqslant 4$. Let $\gamma_{k-3} = x_0x_1 \ldots x_{M(q)-1}$, $\gamma_{k-2} = y_0y_1\ldots y_{2M(q)-1}$, $\gamma_{k-1} = z_0z_1 \ldots z_{2M(q)-1}$. Without loss of generality we can assume that $y_0$, $y_1$ are adjacent to $x_1$, and  $z_0$, $z_1$ are adjacent to $y_0$ (see Fig.~7). In the graph $P- V(\gamma_{k-2})$ we identify successive vertices and edges of the cycle $\gamma_{k-1}$ with successive vertices and edges of the cycle $\gamma_{k-3}$. After the identification we obtain a cycle $\delta = t_0t_1\cdots t_{2M(q)-1}$ and a graph $H \in \cal{P}$ (see Fig.~6). Notice that $\gamma_0$, $\gamma'$ (or~$\gamma_j$ for $1 \leqslant j \leqslant k-3$) are maximal paths (or cycles, respectively) of the same class (say $q$) in $H$. By induction $H$ contains two disjoint and induced even caterpillars $T_{k-2}$ and $S_{k-2}$ which vertices together span all of $H$, and condition $(3)$ holds (for $k$ replaced with $k-2$, and $P$ replaced with $H$). Let $I = \{0 \leqslant i < 2M(q): {t_i \in V(T_{k-2})}\}$ and $J = \{0 \leqslant i < 2M(q): {t_i \in V(S_{k-2})}\}$. We can consider $V_T  = V(T_{k-2})\backslash V(\delta)$ and $V_S  = V(S_{k-2})\backslash V(\delta)$ as sets of vertices in the graph $P$. Hence, the following sets
$$
V_T \cup \{x_{i} : i \in I\} \cup \{z_{i} : i \in I\} \cup \{y_{i} : i \in J\},
$$
$$
V_S \cup \{x_{i} : i \in J\} \cup \{z_{i} : i \in J\} \cup \{y_{i} : i \in I\}
$$
induce (respectively) two disjoint even-caterpillars $T_k$ and $S_k$ which vertices together span all of $P$, and condition $(3)$ holds.
\qed
\end{proof}
\begin{lemma}\label{lemma6.2}
If $m \geqslant 3$, $0 \leqslant a < b \leqslant m$ are integers, and $b-a \geqslant \frac{m}{3} - 1$, then interval $[a, b]$ contains an integer $2^k$ or $m-2^k$ for some integer $k$.
\end{lemma}
\begin{proof}
Let $k$ be integer such that $2^k \leqslant  \frac{m}{3} < 2^{k+1}$. Since $2^{k+1} - 2^{k} \leqslant \frac{m}{3}$ and $(m-2^{k+1}) - 2^{k+1} < \frac{m}{3}$, the interval $[a,b]$ of the length at least $\frac{m}{3}- 1$ contains one of the integers: $$1, 2^k, 2^{k+1}, m-2^{k+1}, m-2^k, m-1,$$
which completes the proof.
\qed
\end{proof}
\begin{theorem}\label{theo6.2}
Let $P\in \cal{P}$. If $M(q)$ is odd and $K(q) \geqslant \frac{M(q)}{3}$ for some $q \in Q$, then $P$ has two disjoint and induced paths which together span $P$.
\end{theorem}
\begin{proof}
Let $(K(q), M(q), S^+(q))$ be a $q$-index-vector of $P$ such that $M(q)$ is odd and $K(q) \geqslant   \frac{M(q)}{3}$. If $M(q) = 1$ ($K(q) = 1$), then the union of two maximal paths of class $q+1$ ($q$, respectively) is a spanning subgraph of $P$. Hence we assume that $M(q) \geqslant 3$ and $K(q) \geqslant 2$. By Lemma 5.2, there exists integer $s$ such that $\max (S^+(q),1) \leqslant s < \min (S^+(q) + K(q), M(q))$ and $|s, M(q)| = 1$. If $S^+(q) = s$, then, by Theorem 3.1(3), $K(q+1) = 1$ and our theorem holds.

Let $l = s - S^+(q) >  0$. Assume that $\gamma_0 = v^0_0v^0_1 \ldots v^0_{M(q)}$, $\gamma'$ are maximal paths of class $q$, and $\gamma_k = v^k_0v^k_1 \ldots v^k_{2M(q)-1}$, $1 \leqslant k < K(q)$, are clock-wise oriented cycles of class $q$ in $P$. Without loss of generality we can assume that vertices
$v^k_1$, $v^k_2$ are adjacent to $v^{k-1}_1$, $1 \leqslant k < K(q)$.
In $P - \bigcup_{0\leqslant k < l}V(\gamma_k)$ we identify successive vertices and edges of the path $v^l_0v^l_1 \ldots v^l_{M(q)}$ with successive vertices and edges of the path $v^l_0 v^l_{2M(q)-1} v^l_{2M(q)-2} \ldots v^l_{M(q)}$.
After the identification we obtain a path $\delta = w_0w_1 \ldots w_{M(q)}$ and a graph $H\in\cal{P}$. Notice that $\delta$ and $\gamma'$ ($\gamma_k$, for $l < k < K(q)$) are two maximal paths ($K(q)-l-1$ cycles, respectively) of the same class (say $q$) in $H$. Hence, $(K_H(q), M_H(q), S_{H}^+(q)) = (K(q) - l, M(q), S_{H}^+(q))$ is the $q$-index-vector of $H$. Let us consider the segment  $v^0_{S^+(q)}v^1_{S^+(q)+1}\cdots v^l_sv^{l+1}_{s+1} \cdots v^{K(q)-1}_{S^+(q)+K(q)-1}v$ of class $q+1$ in $P$, where the vertex $v$ belongs to $\gamma'$. By the definition of $S^+(q)$, $deg_P(v) = 3$. Thus, the segment $w_sv^{l+1}_{s+1}\cdots v^{K(q)-1}_{S^+(q)+K(q)-1}v$ is  of class $q+1$ in $H$ and $deg_H(v) = 3$, whence $S_{H}^+(q) = s$. Therefore, by Theorem 3.1(3),  $K_{H}(q+1)= |s,M(q)|= 1$. Hence, $H$ has two maximal paths $\alpha$ and $\beta$ of class $q+1$ whose vertices together span all of $H$. Let $I = \{0 \leqslant i \leqslant  M(q): w_i \in V(\alpha)\}$ and $J = \{0 \leqslant i \leqslant  M(q): {w_i \in V(\beta)}\}$. We can consider $V_\alpha = V(\alpha)\backslash V(\delta)$ and $V_\beta = V(\beta)\backslash V(\delta)$ as sets of vertices in the graph $P$. Notice that $v_i^lv_i^{l-1} \cdots v_i^0$, $0 \leqslant i \leqslant M(q)$, $v_i^0v_{2M(q)-i}^1  \cdots v_{2M(q)-i}^l$, $0 < i < M(q)$, are segments of class $q-1$ in~$P$. Since a vertex $w_i$ of the path $\delta$ is obtained by the identification of the vertices $v_i^l$ and $v_{2M(q)-i}^l$ in $P$, the following sets
$$
V_{\alpha} \cup \bigcup _{i \in I\cap \{0, M(q)\}} \{v_i^l, v_i^{l-1}, \ldots, v_i^0\}
  \cup \bigcup _{i \in I\backslash \{0, M(q)\}} \{v_i^l \cdots v_i^1v_i^0v_{2M(q)-i}^1  \cdots v_{2M(q)-i}^l\},
$$
and
$$
V_{\beta} \cup \bigcup _{i \in J\cap \{0, M(q)\}} \{v_i^l, v_i^{l-1}, \ldots, v_i^0\}
  \cup \bigcup _{i \in J\backslash \{0, M(q)\}} \{v_i^l \cdots v_i^1v_i^0v_{2M(q)-i}^1  \cdots v_{2M(q)-i}^l\}
$$
induce paths in $P$  whose vertices together span all of $P$.
\qed
\end{proof}

\section{Orbits of non simple graphs in $\cal{P}$}
In the following theorem we characterize orbits of plane triangulations in $\cal{P}$ which are not simple.
\begin{theorem}\label{theo7.1}
$G\in\cal{P}$ is not simple  if and only if $G$ has the orbit of the form
$$
\{(n,1,0), (1,n, n-1), (1,n, 0)\}, \ \hbox{for some integer} \ n > 1 .
$$
\end{theorem}
\begin{proof}
Let $G \in \cal P$. It is easy to prove that the following conditions are equivalent (the last equivalence follows from Theorem~3.1(3--4) and 3.2(2)):
\\[6pt]
(i) $G$ is not simple,
\\[6pt]
(ii) $G$ has a cycle of class $q$ with the length $2$, for some $q \in Q$,
\\[6pt]
(iii) $G \neq K_4$ and it has two edges of class $q$ with ends of degree $3$,  for some $q \in Q$,
\\[6pt]
(iv) $G$ has an  index-vector of the form $(n,1,0)$, for some $n > 1$,
\\[6pt]
(v) $G$ has an orbit of the form $\{(n,1,0), (1,n, n-1), (1,n, 0)\}$, for some $n > 1$.
\\[6pt]
This completes the proof.
\qed
\end{proof}

\end{document}